\documentclass[12pt,reqno]{amsart}
\usepackage{graphicx}
\usepackage{psfrag}
\usepackage{mathrsfs}
\usepackage{color}
\usepackage{amsmath,amssymb}
\font\de=cmssi12

\numberwithin{equation}{section}

\newcommand{\Per}{{\rm Per}}
\newcommand{\ep} {\varepsilon} \newcommand{\eps} {\varepsilon}
\newcommand{\s} {\sigma}
\newcommand{\ti}{\pitchfork}
\newcommand{\PW}{W}

\def \ZZ {{\mathbb Z}}
\def \RR {{\mathbb R}}
\newcommand{\PB}{\cR_{\eps,l}}

\def \cL {{\mathcal L}}
\def \cR {{\mathcal R}}
\def \cS {{\mathcal S}}
\def \cT {{\mathcal T}}
\def \cU {{\mathcal U}}
\def \cW {{\mathcal W}}

\newcommand \W {\cW}

\newtheorem{maintheorema}{Theorem}

\newtheorem{theorem}{Theorem}[section]
\newtheorem{conjecture}[theorem]{Conjecture}
\newtheorem{corollary}[theorem]{Corollary}
\newtheorem{proposition}[theorem]{Proposition}
\newtheorem{lemma}[theorem]{Lemma}
\newtheorem{definition}[theorem]{Definition}

\theoremstyle{remark}
\newtheorem{remark}[theorem]{Remark}

\begin{document}


\author{F. Rodriguez Hertz}
\address{IMERL-Facultad de Ingenier\'\i a\\ Universidad de la
Rep\'ublica\\ CC 30 Montevideo, Uruguay.}
\email{frhertz@fing.edu.uy}\urladdr{http://www.fing.edu.uy/$\sim$frhertz}

\author{M. A. Rodriguez Hertz}
\address{IMERL-Facultad de Ingenier\'\i a\\ Universidad de la
Rep\'ublica\\ CC 30 Montevideo, Uruguay.}
\email{jana@fing.edu.uy}\urladdr{http://www.fing.edu.uy/$\sim$jana}

\author{A. Tahzibi}
\address{Departamento de Matem\'atica,
  ICMC-USP S\~{a}o Carlos, Caixa Postal 668, 13560-970 S\~{a}o
  Carlos-SP, Brazil.}
\email{tahzibi@icmc.usp.br}\urladdr{http://www.icmc.usp.br/$\sim$tahzibi}

\author{R. Ures}
\address{IMERL-Facultad de Ingenier\'\i a\\ Universidad de la
Rep\'ublica\\ CC 30 Montevideo, Uruguay.} \email{ures@fing.edu.uy}
\urladdr{http://www.fing.edu.uy/$\sim$ures}

\thanks{This work was done in Universidad de la Rep\'{u}blica- Uruguay and ICMC-USP, S\~{a}o Carlos, Brazil.
 AT thanks CNPq and FAPESP for financial support and Universidad de la Rep\'{u}blica for warm hospitality and  financial support. FRH, MRH and RU acknowledge warm hospitality of ICMC-USP
 and financial support of CNPq and FAPESP}



\renewcommand{\subjclassname}{\textup{2000} Mathematics Subject Classification}

\date{\today}

\title{New criteria for ergodicity and non-uniform hyperbolicity}

\begin{abstract}
In this work we obtain a new criterion to establish ergodicity and non-uniform hyperbolicity of smooth measures of
diffeomorphisms. This method allows us to give a more accurate description of certain ergodic components. The use of this criterion in combination with topological devices such as blenders lets us obtain global ergodicity and abundance of non-zero Lyapunov exponents in some contexts.
\par
In the partial hyperbolicity context, we obtain that stably ergodic diffeomorphisms are $C^1$ dense among volume preserving partially hyperbolic diffeomorphisms with two-dimensional center bundle. This is motivated by a well-known conjecture of C. Pugh and M. Shub.
\end{abstract}

\maketitle

\section{Introduction}
Our work is concerned with the ergodic properties of smooth invariant measures for diffeomorphisms under some hypotheses about their Lyapunov exponents. \par
Among the few techniques that are used to establish ergodicity of a smooth measure, we can mention the Hopf argument, which we shall discuss in sections \S \ref{ergodic.homoclinic.classes} and \S \ref{section.criterion}, harmonic analysis (see for instance \cite{katokhasselblatt},\cite{manhe}), and case by case studies, such as in the Anosov-Katok examples \cite{anosovkatok}.\par
Along the line of the Hopf argument, there is the local ergodicity criterion, which consists in showing that there is an open set contained $(\hspace*{-.6em}\mod 0)$ in an er\-go\-dic component. Here we present criteria which can be seen as complementary to local ergodicity.
\subsection{Ergodic homoclinic classes}\label{ergodic.homoclinic.classes}
E. Hopf proved that the geodesic flow of a surface of negative curvature is ergodic by showing that any measurable set that is invariant under the geodesic flow is also invariant by the stable and unstable distributions of the flow \cite{hopf1939}. This technique, later known as the {\de Hopf argument}, was also used by D. Anosov \cite{anosov}, D. Anosov and Ya. Sinai \cite{Anosov-Sinai} to prove ergodicity of Anosov systems. It was further used by Ya. Pesin \cite{pesin1977} in the context of non-zero Lyapunov exponents, to prove that a hyperbolic measure has only countably many ergodic components (see Theorem \ref{pesin.decomposition} below). The Hopf argument has since then become a standard argument to prove ergodicity.\par
Here we present a refinement of the Hopf argument which provides a more accurate description of some ergodic components of a smooth invariant measure. Indeed, we introduce the concept of ergodic homoclinic classes associated to a hyperbolic point (see below). Under mild hypotheses, these sets turn out to be hyperbolic ergodic components. Moreover, in the context of Pesin's Ergodic Component Theorem (Theorem \ref{pesin.decomposition}), these ergodic homoclinic classes depict all ergodic components and increase the resemblance between Pesin's Theorem and Smale's Spectral Decomposition Theorem (see below).\par
However, our new criterion, opposite to Pesin's Theorem, does not require absence of zero Lyapunov exponents a priori, but only requires that certain sets associated to a hyperbolic periodic point be of positive measure. Ergodicity and non-uniform hyperbolicity will follow as a consequence.\par
Let $f:M\to M$ be a $C^{1+\alpha}$ diffeomorphism of a closed Riemannian manifold $M$. Given a hyperbolic periodic point $p$, let us define the {\de ergodic homoclinic class} of $p$, $\Lambda(p)$, as the set of regular points $x\in M$ such that\par
\begin{figure}[h]
\hspace*{-1cm}\begin{minipage}{7cm}
\includegraphics[bb=91 264 543 599, width=6cm]{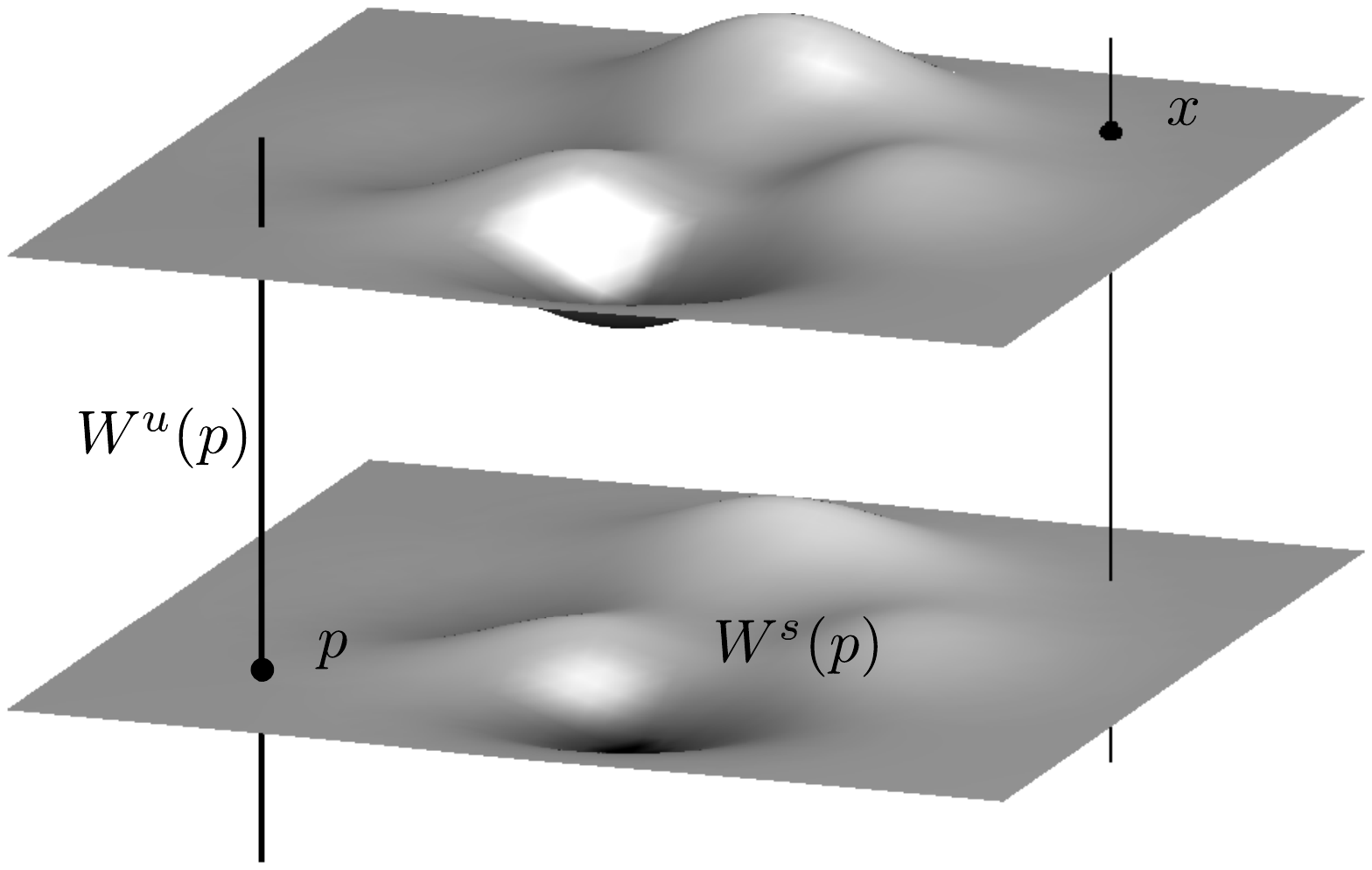}
\end{minipage}\hspace*{1cm}
\begin{minipage}{5cm}
\begin{eqnarray}\label{Bup}W^s(o(p))\pitchfork W^u(x)\ne\emptyset&\quad&\\
\label{Bsp}W^u(o(p))\pitchfork W^s(x)\ne\emptyset&\quad&
\end{eqnarray}
\end{minipage}
\caption{\label{pesin.homoclinic.class} $x\in\Lambda(p)$, the ergodic homoclinic class of $p$}
\end{figure}
Here $W^s(x)$ is the {\de stable Pesin manifold} of $x$, that is,
$$W^s(x)=\left\{y\in M: \lim\sup_{n\to +\infty}\frac{1}{n}\log d(f^n(x),f^n(y))<0\right\}$$
and $W^u(x)$ is the {\de unstable Pesin manifold} of $x$, that is, the stable Pesin manifold of $x$ for $f^{-1}$.
For almost every point, stable and unstable Pesin manifolds are immersed manifolds, see \S\ref{subsection.pesin} and references therein. $W^s(o(p))$ is the stable Pesin manifold of the orbit of $p$. The ergodic homoclinic classes are invariant sets. In fact, if $m(\Lambda(p))>0$, then $\Lambda(p)$ is an ergodic component, see Theorem \ref{main.theorem} and \S \ref{section.criterion}.\par
Now note that we can write an ergodic homoclinic class as the intersection of two invariant sets:
$$\Lambda(p)=\Lambda^s(p)\cap \Lambda^u(p)$$
where $\Lambda^u(p)$ is the set of regular $x$ satisfying relation (\ref{Bup}),
 and $\Lambda^s(p)$ is the set of regular $x$ satisfying relation (\ref{Bsp}).
  $\Lambda^s(p)$ is $s$-saturated, and $\Lambda^u(p)$ is $u$-saturated.\par
We prove the following result:
\begin{maintheorema}\label{main.theorem}
Let $f:M\to M$ be a $C^{1+\alpha}$ diffeomorphism over a Riemannian manifold $M$, and let $m$ be a smooth invariant measure over $M$. If $m(\Lambda^s(p))>0$ and $m(\Lambda^u(p))>0$ for some hyperbolic periodic point, then $$\Lambda(p)\circeq \Lambda^s(p)\circeq \Lambda^u(p)$$ and $f$ is ergodic on $\Lambda(p)$. Moreover, $f$ is non-uniformly hyperbolic on $\Lambda(p)$.
\end{maintheorema}
This improves the former description of ergodic components, and completes the parallelism between Smale's Spectral Decomposition Theorem and Pesin's Ergodic Component Theorem. Indeed, let us recall that
\begin{theorem}[Smale \cite{smale67}] Let $f: M\to M$ be a diffeomorphism such that $\overline{\Per(f)}$ is hyperbolic. Then:
\begin{enumerate}
\item there are disjoint compact invariant sets $\Lambda_i$ such that $f|_{\Lambda_i}$ is transitive and
$$\overline{\Per(f)}=\Lambda_1\cup\dots\cup\Lambda_n.$$
\item Each $\Lambda_i$ decomposes as $k_i$ compact disjoint sets permuted by $f$, on each of which $f^{k_i}$ is topologically mixing and semi-conjugated to a sub-shift of finite type.
\item There exist hyperbolic periodic points $p_i$ such that $\Lambda_i=\overline{H(p_i)}$, where $H(p_i)$ is the homoclinic class of $p_i$.
\end{enumerate}
\end{theorem}
 The sets $\Lambda_i$ in the theorem above are called {\de basic sets}, and are the closure of homoclinic classes of periodic points. The homoclinic class of a periodic point $p$ is the set of periodic points that are {\de homoclinically related} to $p$, that is, the set of {\it periodic points} $x$ that satisfy relations (\ref{Bup}) and (\ref{Bsp}).\par
A combination of Pesin's Ergodic Component Theorem \cite{pesin1977}, Katok's Closing Lemma \cite{katok1980} and our Theorem \ref{main.theorem} above, gives an ergodic analogous to Smale's Spectral Decomposition Theorem:

\begin{theorem}[Pesin's Ergodic Component Theorem \cite{pesin1977}]\label{pesin.decomposition} Let $f:M\to M$ be a $C^{1+\alpha}$ diffeomorphism and $m$ a smooth measure that is hyperbolic over an invariant set $V$. Then:
 \begin{enumerate}
\item $$V\circeq \Lambda_1\cup\dots\cup\Lambda_n\cup\dots$$
where $\Lambda_i$ are disjoint measurable invariant sets such that $f|_{\Lambda_i}$ is ergodic.
\item Each $\Lambda_i$ decomposes as $k_i$ disjoint sets permuted by $f$, on each of which $f^{k_i}$ is mixing and isomorphic to a sub-shift of finite type (Bernoulli).
\item There exist hyperbolic periodic points $p_i$ such that $\Lambda_i=\Lambda(p_i)$
\end{enumerate}
 \end{theorem}
 \par
In fact, item (3) is a consequence of the following: the Closing Lemma of A. Katok \cite{katok1980} provides ``visible" ergodic components with hyperbolic periodic points. Our criterion (Theorem \ref{main.theorem}) then implies that these ergodic components coincide with ergodic homoclinic classes.\par
Let us note that A. Tahzibi already used hyperbolic periodic points to prove ergodicity \cite{ta}. However, in his work, he required absence of zero Lyapunov exponents. \par
Finally, one could ask if the hypotheses of Theorem \ref{main.theorem} are sharp, that is: could we obtain ergodicity by simply asking, for instance, that $m(\Lambda^s(p))>0$? The answer is negative: in \S\ref{subsection.example} we present an example of a non-ergodic diffeomorphism with $m(\Lambda^s(p))=1$ and $m(\Lambda^u(p))=0$.
\subsection{A conjecture of C. Pugh and M. Shub concerning stable ergodicity}
A $C^{1+\alpha}$ volume preserving diffeomorphism $f$ is {\de stably ergodic} if there exists
a $C^1$ neighborhood $\cU$ of $f$ such that all $C^{1+\alpha}$ volume preserving diffeomorphisms in $\cU$ are
ergodic.\par
Until 1994, the only known examples of stably ergodic diffeomorphisms were Anosov diffeomorphisms. Indeed, the ergodic theory of uniformly hyperbolic diffeomorphisms has been extensively studied, beginning with the works of D. Anosov, Ya. Sinai, D. Ruelle and R. Bowen.\par
In 1994, M. Grayson, C. Pugh and M. Shub gave the first example of a
   non-hyperbolic diffeomorphism that is stably ergodic with respect to a smooth
   measure \cite{gps}. Their example belongs to a broad class of
   dynamics which is called partially hyperbolic.
 A diffeomorphism
$f:M\rightarrow M$ is partially hyperbolic if $TM$ splits into
three invariant bundles: one that is contracting, one that is expanding, and a third one, the center bundle, which has an intermediate behavior (see
Section \ref{preliminaries} for a precise definition).
For instance, the time-one map of the geodesic flow on the unit
 bundle of a compact manifold with negative sectional curvature is partially hyperbolic.\par
 In 1995, C. Pugh and M. Shub posed the following:
 \begin{conjecture}
 [C. Pugh, M. Shub \cite{pughshub1995,pughshub2000}] Stable ergodicity is $C^r$-dense among partially hyperbolic diffeomorphisms.
 \end{conjecture}
In the last years, many advances have been made in the direction of this conjecture, a survey of which can be found in \cite{rru2}. In particular, we
want to mention the very recent works: by K. Burns and A. Wilkinson
\cite{burns-wilkinson2006} proving that (essential) accessibility
plus a bunching condition implies ergodicity, and by F. Rodriguez Hertz, M. Rodriguez Hertz, and
R. Ures \cite{rhrhu2006} obtaining that stable ergodicity is $C^\infty$-dense among partially hyperbolic diffeomorphisms with one-dimensional center bundle.\par
Many of the arguments of \cite{burns-wilkinson2006} and
\cite{rhrhu2006} seem to be hard to
generalize. Apparently, proving the Pugh-Shub
Conjecture for center bundle of any dimension will require dramatically new techniques.
Indeed, some center bunching condition seems to be needed to obtain that the
holonomies of the strong foliations when restricted to center
manifolds are Lipschitz. This fact is used in an essential way, for instance, in \cite{burns-wilkinson2006}. Also, one-dimensionality of the center bundle has been crucial in the arguments of \cite{rhrhu2006}.
\par
In this paper we try a different approach that allows us to
prove the $C^1$-denseness of stable ergodidicity among partially hyperbolic diffeomorphisms with two-dimensional center bundle:
\begin{maintheorema}[Pugh-Shub Conjecture]\label{teorema.prueba.conjetura}
Stable ergodicity is $C^1$-dense among partially hyperbolic diffeomorphisms with 2-dimensional center bundle.
\end{maintheorema}
Moreover, given any $C^r$ volume preserving diffeomorphism it is possible to make a $C^1$-perturbation in order to find a $C^r$ diffeomorphism which is stably ergodic. We stress that $C^{1+\alpha}$ regularity condition for the systems under consideration is important for the existence of stable manifolds or the absolute continuity of the stable lamination \cite{pesin1977, pughshub1989}. Let us mention that recently, A. Avila, J. Bochi and A. Wilkinson have proved the $C^1$-density of ergodic diffeomorphisms among symplectic ones \cite{avilabochiwilkinson2008}. It is still an open question if stable ergodicity is dense among partially hyperbolic symplectic diffeomorphisms.\par
\par
Theorem \ref{teorema.prueba.conjetura} implies that ergodicity is $C^1$-generic among partially hyperbolic maps with two-dimensional center bundle. This is due to an outstanding result of A. Avila, proving the $C^1$-denseness of smooth diffeomorphisms among $C^1$-volume preserving diffeomorphisms \cite{avila2008}. \par
In the proof of Theorem \ref{teorema.prueba.conjetura}, the Lyapunov exponents in the center directions play a significant role (see next section). An antecedent result taking advantage of the interplay between
partial hyperbolicity and non-zero Lyapunov exponents is the work of K. Burns, D. Dolgopyat and Ya. Pesin \cite{burns-dolgopyat-pesin2002}: they prove that if a partially hyperbolic $f$ is accessible and the center Lyapunov exponents have the same sign on a positive
measure set then $f$ is stably ergodic (See Theorem \ref{BDP}).
The question that
naturally arises is: what happens when the exponents are not zero
but the signs of central Lyapunov exponents are different? In this paper we use our criterion in order to prove the ergodicity of diffeomorphisms in the case of mixed sign center Lyapunov exponents. The criterion is used in combination with the blenders introduced by C. Bonatti and L. Diaz in a very different context \cite{bodi1996}, in order to create one single ergodic component which is in some sense robust. This produces stable ergodicity. \par
In next section, we give a sketch of the proof of Theorem \ref{teorema.prueba.conjetura}. Section \ref{preliminaries} introduces some preliminary concepts and results. Section \ref{section.criterion} is devoted to proving Theorem \ref{main.theorem}. It has an introductory part explaining the proof. In Section \ref{subsection.example}, an example is presented, showing that $m(\Lambda^s(p))>0$ is not enough to guarantee that $\Lambda^s(p)$ is an ergodic component.
In Section \ref{section.prueba.conjetura} we prove Theorem \ref{teorema.prueba.conjetura}. Its proof is split into two theorems: Proposition \ref{theorem.approximation} is proved in Section \ref{subsection.approximation}, and Theorem \ref{teorema.establemente.ergodico} is proved in Section \ref{subsection.establemente.ergodico}.
%
\section{Sketch of the proof of the Pugh-Shub conjecture}\label{section.sketch.pugh.shub}
In this section we sketch the proof of Theorem \ref{teorema.prueba.conjetura}, the concepts that appear in italics are defined and explained in Subsections \S\ref{subsection.partialdominated} and \S\ref{subsection.blenders}.\par
Let $f$ be a $C^{1+\alpha}$ {\em partially hyperbolic} diffeomorphism with two-dimensional center bundle, preserving a smooth measure $m$. We want to show that $f$ can be $C^1$-approximated by a stably ergodic diffeomorphism. \par
First of all, let us note that $f$ can be $C^1$-approximated by a {\em stably accessible} diffeomorphism, due to a result by D. Dolgopyat and A. Wilkinson \cite{dolgopyat-wilkinson2003}. So, we may assume that $f$ is stably accessible. As a consequence, every invariant set with positive measure is, in fact, dense.\par
The proof is split into cases, according to the signs of the {\em center Lyapunov exponents}, $\lambda^+$ and $\lambda^-$. The following remarkable theorem rules out two important cases:
\begin{theorem}[K. Burns, D. Dolgopyat, Ya. Pesin \cite{burns-dolgopyat-pesin2002}]\label{BDP} Let $f$ be a $C^{1+\alpha}$ partially hyperbolic diffeomorphism preserving a smooth measure $m$. Assume that $f$ is accessible and has negative center Lyapunov exponents on a set of positive measure. Then $f$ is stably ergodic.
\end{theorem}
In the case that $f$ is not in the hypotheses of Theorem \ref{BDP} (center Lyapunov exponents are both positive or both negative), then by a combination of the results of A. Baraviera and C. Bonatti \cite{baraviera-bonatti2003}, and J. Bochi and M. Viana \cite{bochi-viana2005}, we can make a $C^1$ perturbation to obtain the following:
(a) the center bundle has a {\em dominated splitting} $E^c=E^+\oplus E^-$, and (b) the sets $C^+$ where $\lambda^+(x)>0$, and $C^-$ where $\lambda^-(x)<0$ have positive measure. \par
Note that $m(C^+\cup C^-)=1$. Indeed, if $x$ does not belong to $C^-$, then by definition $\lambda^-(x)\geq 0$. The domination of the center splitting then implies that $\lambda^+(x)>0$, hence $x\in C^+$.\par
The main novelty here is the use of {\em blenders} in order to show that $C^+\subset \Lambda^u(p_1)$ and $C^-\subset\Lambda^s(p_1)$ for some periodic point $p_1$, for all diffeomorphisms in a neighborhood of the perturbation of $f$. Theorem \ref{main.theorem} then applies and proves that this perturbation is stably ergodic. \par
Now, we can assume that neither $E^+$ nor $E^-$ in the dominated splitting above are hyperbolic. Otherwise, we could consider $f$ as a partially hyperbolic diffeomorphism with one-dimensional center bundle. In this case, $f$ is known to be approximated by stably ergodic diffeomorphisms \cite{burns-wilkinson2006,rhrhu2006}.\par
Using the Ergodic Closing Lemma for the conservative setting \cite{arnaud1998} and the Conservative Franks' Lemma, we obtain a $C^1$-perturbation with three hyperbolic periodic points $p_0$, $p_1$ and $p_2$ of {\em unstable indices}, $u$, $(u+1)$ and $(u+2)$, where $u=\dim E^u$.
Our goal is to prove that for some perturbation the sets $C^-$ and $C^+$ are robustly included in $\Lambda(p_1)$.\par
Using ideas of C. Bonatti and L. Diaz \cite{bonatti-diaz2006} adapted to the conservative setting \cite{rhrhtu2009}, we obtain two {\it blenders}, ${\rm Bl}^{cu}(p_1)$ and ${\rm Bl}^{cs}(p_1)$, associated to $p_1$ (see Theorem \ref{proposition.blender.conservative}). The main property of a {\em $cu$-blender} is that any {\em $(u+1)$-strip} which is {\em well placed} in ${\rm Bl}^{cu}(p_1)$ will transversely intersect $W^s(p_1)$, and of a {\em $cs$-blender} is that
any {\em $(s+1)$-strip} which is {\em well placed} in ${\rm Bl}^{cs}(p_1)$ will transversely intersect $W^u(p_1)$.
Moreover, this property of blenders is $C^1$-robust.\par
Now, if a point $x$ is in $C^+$, then for some suitable iterate, its unstable Pesin manifold is a $(u+1)$-strip which is well placed in ${\rm Bl}^{cu}(p_1)$, so $x$ will belong to $\Lambda^u(p_1)$. Hence, $C^+\subset \Lambda^u(p_1)$. Analogously, one will obtain that $C^-\subset \Lambda^s(p_1)$. The ergodicity criterion introduced in Theorem \ref{main.theorem} now applies, and we obtain that $\Lambda(p)\circeq M$, so the diffeomorphism is ergodic. But, as we have mentioned, this intersection property of blenders is $C^1$-robust, so this new diffeomorphism is in fact stably ergodic. In this way we have obtained a $C^1$-perturbation of $f$ which is stably ergodic.\par
Let us state this result more precisely:
\begin{maintheorema}\label{teorema.establemente.ergodico}
Let $f$ be a $C^{1+\alpha}$ partially hyperbolic diffeomorphism preserving a smooth measure, and such that its center bundle is two-dimensional. Assume $f$ satisfies the following properties:
 \begin{enumerate}
 \item $f$ is {\em accessible},
 \item  $E^c = E^- \oplus E^+$ admits a nontrivial dominated splitting
  \item $\int_M \lambda^- d m < 0 $ and $\int_M \lambda^+ d m > 0,$
 \item $f$ admits a $cs$-blender ${\rm Bl}^{cs}(p)$ and a $cu$-blender ${\rm Bl}^{cu}(p)$ associated to a hyperbolic periodic point $p$ of stable index $(s+1)$.
  \end{enumerate}
  then $f$ is stably ergodic.
\end{maintheorema}

Note that the above theorem, like Theorem \ref{BDP}, is not a generic result, but gives precise conditions under which a diffeomorphism is stably ergodic. The scheme of the  proof of Theorem \ref{teorema.prueba.conjetura} is: either $f$ is in the hypotheses of Theorem \ref{BDP}, or else one can make a $C^1$-perturbation, so that $f$ is in the hypotheses of Theorem \ref{teorema.establemente.ergodico}. In either case, $f$ will be stably ergodic.\par
To what extent do these arguments apply? So far, with the techniques presented in this work it can be proved a Pugh-Shub Conjecture for some type of center bundles: Stable ergodicity is $C^1$-dense among partially hyperbolic diffeomorphisms for which the center bundle admits a dominated splitting $E^c=E_1\oplus E_2\oplus E_3$, with $\dim E_1=\dim E_3=1$, and $E_1$, $E_3$ are non-hyperbolic. The proof of this result will appear elsewhere.
\section{Preliminaries}\label{preliminaries}

\subsection{Non-uniform hyperbolicity}\label{subsection.pesin}
Let us review some results about Pesin theory that shall be used in this paper. A good summary of these facts may be found, for instance, in
\cite{pughshub1989} and \cite{ledrappieryoung1985}. For further references, see
A. Katok's paper \cite{katok1980} and  the book by L. Barreira and Ya. Pesin \cite{barreira-pesin}.\par
Let $f:M\to M$ be a $C^1$ diffeomorphism  of a compact Riemannian manifold of dimension $n$.
Given a vector $v\in T_xM$, let the {\de Lyapunov exponent of $v$} be the exponential growth rate of $Df$ along $v$, that is
\begin{equation}\label{lyap exp}
\lambda(x,v)=\lim_{|n|\to\infty}\frac1n\log|Df^n(x)v|
\end{equation}
in case this amount is well defined. And let $E_\lambda(x)$ be the subspace of $T_xM$ consisting of all $v$ such that the Lyapunov exponent of $v$ is $\lambda$. Then we have the following:
\begin{theorem}[V. Osedelec]
For any $C^1$ diffeomorphism $f:M\to M$ there is an $f$-invariant Borel set $\cR$ of total probability (in the sense that $\mu(\cR)=1$ for all invariant probability measures $\mu$), and for each $\eps>0$ a Borel function $C_\eps: \cR\to(1,\infty)$
such that for all $x\in \cR$, $v\in T_xM$ and $n\in\ZZ$
\begin{enumerate}
\item $T_xM=\bigoplus_\lambda E_\lambda(x)$ ({\de Oseledec's splitting})
\item For all $v\in E_\lambda(x)$
$$C_\eps(x)^{-1}exp[(\lambda-\eps)n]|v|\leq|Df^n(x)v|\leq C_\eps(x)exp[(\lambda+\eps)n]|v|$$
\item $\angle\left(E_\lambda(x),E_{\lambda'}(x)\right)\geq C_\eps(x)^{-1}$ if $\lambda\ne\lambda'$
\item $C_\eps(f(x))\leq \exp(\eps) C_\eps(x)$
\end{enumerate}
\end{theorem}
The set $\cR$ is called the set of {\de regular points}. For simplicity, we will assume that all points in $\cR$ are Lebesgue density points.
We also have that $Df(x)E_\lambda(x)=E_\lambda(f(x))$. If
an $f$-invariant measure $\mu$ is ergodic then the Lyapunov exponents and $\dim E_\lambda(x)$ are
constant $\mu$-a.e.\par
For fixed $\ep > 0$ and given $l > 0,$ we define the {\de Pesin blocks}:
 $$
  \PB = \left\{ x \in \mathcal{R} : C_\eps(x) \leq l\right\}.
 $$
Note that Pesin blocks are not necessarily invariant. However $ f(\PB) \subset \cR_{ \eps,\exp(\eps)l}$.
Also, for each $\eps>0$, we have
 \begin{equation}\label{pesin region}
\cR = \bigcup_{ l= 1}^{\infty} \PB
\end{equation}
We loose no generality in assuming that $\PB$ are compact.
For all $x\in \cR$ we have
$$
T_xM=\bigoplus_{\lambda<0} E_\lambda(x)\oplus E^0(x)\bigoplus_{\lambda>0} E_\lambda(x)
$$
where $E^0(x)$ is the subspace generated by the vectors having zero Lyapunov exponents. Let $\mu$ be an invariant measure. When $E^0(x)=\{0\}$ for all $\mu$-a.e. $x$ in a set $N$, then we say that $f$ is {\de non-uniformly hyperbolic} on $N$ and that $\mu$ is a {\de hyperbolic measure} on $N$.\par
Now, let us assume that $f\in C^{1+\alpha}$ for some $\alpha>0$. Given a regular point $x$, we define its
{\de stable Pesin manifold} by
\begin{equation}\label{stable Pesin manifold}
\PW^s(x)=\left\{y:\limsup_{n\to +\infty}
\frac{1}{n}\log  d(f^n(x),f^n(y))<0\right\}
\end{equation}
The {\de unstable Pesin manifold} of $x$, $\PW^u(x)$ is the stable Pesin manifold of $x$ with respect to $f^{-1}$.
Stable and unstable Pesin manifolds of points in $\cR$ are immersed manifolds \cite{pesin1977}. We stress that $C^{1+\alpha}$ regularity is crucial for this to happen. In this way we obtain a partition
 $x\mapsto\PW^s(x)$ , which we call {\de stable partition}. Unstable partition is defined analogously. Stable and unstable partitions are invariant.\par
On the Pesin blocks we have a continuous variation: Let us call $W^s_{loc}(x)$ the connected component of $W^s(x)\cap B_r(x)$ containing $x$, where $B_r(x)$ denotes the Riemannian ball of center $x$ and radius $r>0$, which is sufficiently small but fixed. Then
\begin{theorem}
[Stable Pesin Manifold Theorem \cite{pesin1977}]
Let $f:M\to M$ be a $C^{1+\alpha}$ diffeomorphism preserving a smooth measure $m$. Then, for each $l>1$ and small $\eps>0$, if $x\in \cR_{\eps,l}$:
\begin{enumerate}
\item $W^s_{loc}(x)$ is a disk such that $T_xW^s_{loc}(x)=\bigoplus_{\lambda<0} E_\lambda(x)$
\item $x\mapsto\PW^s_{loc}(x)$ is continuous over $\PB$ in the $C^1$ topology
\end{enumerate}
\end{theorem}
In particular, the dimension of the disk $W^s_{loc}(x)$ equals the number of negative Lyapunov exponents of $x$. An analogous statement holds for the unstable Pesin manifold.
\subsection{Absolute continuity}
An important notion behind the criterion we are going to prove is absolute continuity. Let us state the definitions we will be using. The point of view we follow is similar to that in \cite{ledrappieryoung1985}.\par
Let $\xi$ be a partition of the manifold $M$. We shall call $\xi$ a {\de measurable partition} if the quotient space $M/\xi$ is separated by a countable number of measurable sets. For instance, the partition of the 2-torus by lines of irrational slope is not measurable, while the partition of $[0,1]$ by singletons is measurable. The quotient space $M/\xi$ of a Lebesgue space $M$ by a measurable partition $\xi$ is again a Lebesgue space \cite{rohlin}.\par
Associated to each measurable partition $\xi$ of a Lebesgue space $(M,{\mathcal B},m)$ there is a canonical system of {\de conditional measures} $m_x^\xi$, which are Lebesgue measures on $\xi(x)$, the element of $\xi$ containing $x$, and with the property that for each $A\in{\mathcal B}$ the set $A\cap\xi(x)$ is measurable in $\xi(x)$ for almost all $\xi(x)$ in $M/\xi$, and the function $x\mapsto m_x^\xi(A\cap\xi(x))$ is measurable, with:
\begin{equation}\label{ecuacion.measurable.partition}m(A)=\int_{M/\xi}m_x^\xi(A\cap\xi(x))dm_T\end{equation}
where $m_T$ is the quotient measure on $M/\xi$. For each measurable partition this canonical system of conditional measures is unique (mod $0$), i.e. any other system is the same for almost all $\xi(x)\in M/\xi$. Conversely, if there is a canonical system for a partition, then the partition is measurable. In our case, we will be interested in stable and unstable partitions, note that in general these partitions are not measurable.\par
A measurable partition $\xi$ is {\de subordinate} to the unstable partition $W^u$ if for $m$-a.e. $x$ we have $\xi(x)\subset W^u(x)$, and $\xi(x)$ contains a neighborhood of $x$ which is open in the topology of $W^u(x)$.
\begin{definition}\label{fubini.like}
  $m$ has {\de absolutely continuous conditional measures on unstable manifolds} if for every measurable partition $\xi$ subordinate to $W^u$, $m^\xi_x<<\lambda^u_x$ for $m$-a.e. $x$, where $\lambda^u_x$ is the Riemannian measure on $W^u(x)$ given by the Riemannian structure of $W^u(x)$ inherited from $M$.
\end{definition}
Now, take a point $x_0\in \cR$, the set of regular points. Assume that $x_0$ has at least a negative Lyapunov exponent. Take two small disks $T$ and $T'$ near $x_0$ which are transverse to $W^s(x_0)$.
Then we can define the {\de holonomy map} with respect to these transversals as a map $h$ defined on a subset of $T$ such that $h(x)=W^s_{loc}(x)\cap T'$.
The domain of $h$ consists of the points $x\in T\cap \cR$ whose unstable manifold have the same dimension as $W^s(x_0)$, and which transversely intersect $T$ and $T'$. $h$ is a bijection.
\begin{definition}\label{absolute.continuity}
We say that the unstable partition is {\de absolutely continuous} if all holonomy maps are measurable and take Lebesgue zero sets of $T$ into Lebesgue zero sets of $T'$.
\end{definition}
Absolute continuity of the stable partition is defined analogously.
\begin{theorem}[Ya. Pesin \cite{pesin1977}, C. Pugh, M. Shub \cite{pughshub1989}]\label{teorema.cont.abs.Pesin}
The stable and unstable partitions are absolutely continuous.
\end{theorem}
Note that the holonomy maps are continuous when restricted to the Pesin blocks $\PB$. Also, Theorem \ref{teorema.cont.abs.Pesin} implies there are measurable partitions subordinate to $W^s$ and $W^u$ for which a Fubini-like property like (\ref{ecuacion.measurable.partition}) applies.
\subsection{Partial hyperbolicity and dominated splitting} \label{subsection.partialdominated}
Let $M$ be a compact Riemannian manifold. A diffeomorphism $f:M\to M$ is {\de partially hyperbolic} on an $f$-invariant set $\Lambda$
if it admits a non-trivial $Df$-invariant splitting of
the tangent bundle $T_\Lambda M = E^s\oplus E^c \oplus E^u$, such that all
unit vectors $v^\s\in E^\s_x$ ($\s= s, c, u$) with $x\in \Lambda$ verify:

$$\|Df(x)v^s\| < \|Df(x)v^c\| < \|Df(x)v^u\| $$

for some suitable Riemannian metric. We require that both $E^s$ and $E^u$ be non trivial. We also require that
$$\|Df|_{E^s}\| < 1\qquad\mbox{and}\qquad\|Df^{-1}|_{E^u}\| < 1$$
We say that $f$ is partially hyperbolic if $\Lambda=M$. If $E^c$ is the trivial bundle on $\Lambda$, we say that $\Lambda$ is a {\de hyperbolic set}. If $\Lambda$ is hyperbolic and $f$ is transitive on $\Lambda$, we call $\dim E^s$ the {\de stable index} of $\Lambda$, and $\dim E^u$ the {\de unstable index} of $\Lambda$. \par
 It is a known fact that there are foliations $\W^{ss}$ and $\W^{uu}$ tangent to the distributions $E^s$ and $E^u$ respectively (see for instance \cite{bp}). A set $X$ will be
called {\de $s$-saturated} if it is a union of leaves of $\W^{ss}$. {\de $u$-saturated} sets are defined analogously.
The {\de accessibility class} of the point $x\in
M$ is the minimal set containing $x$ such that it is both $s$- and $u$-saturated. The
diffeomorphism $f$ has the {\de accessibility property} if the
accessibility class of some $x$ is $M$. It has the {\de essential
accessibility property} if every measurable set which is both $s$-saturated and
$u$-saturated has null or full measure. Obviously, accessibility property implies essential accessibility property.\par
A diffeomorphism $f$ is {\de stably accessible} if there is a $C^1$-neighborhood of $f$ composed by accessible diffeomorphisms.
\begin{theorem}[D. Dolgopyat, A. Wilkinson \cite{dolgopyat-wilkinson2003}]\label{teorema.stable.accessibility} Stable accessibility is $C^1$-dense among partially hyperbolic diffeomorphisms.
\end{theorem}
Accessibility has the following interesting consequence:
\begin{theorem}[K. Burns, D. Dolgopyat, Ya. Pesin \cite{burns-dolgopyat-pesin2002}]\label{teorema.accesible.orbita.densa}
Let $f$ be a partially hyperbolic diffeomorphism preserving a smooth measure, such that $f$ has the accessibility property. Then for each $\eps>0$ there is a $C^1$-neighborhood $\cU$ such that for all $g$ in $\cU$, a.e. orbit is $\eps$-dense. This means, for each $\eps$-ball $B$ a.e. orbit enters the $B$.
\end{theorem}
\begin{remark}\label{remark.accessibility.dense}
If $f$ has the essential accessibility property, then almost every orbit is dense \cite{burns-dolgopyat-pesin2002}. In fact, if for $m$-almost every point its accessibility class is $\eps$-dense, then $m$-almost every orbit is $\eps$-dense. See also \cite{rru2}.
\end{remark}
\begin{remark}\label{remark.eps.accessibility}
If $f$ has the accessibility property, then for every $\eps>0$ there exists a $C^1$-neighborhood $\cU$ of $f$ such that if $g\in\cU$, then all accessibility classes of $g$ are $\eps$-dense.
\end{remark}
The leaf of $\W^\s$ containing $x$ will be called $W^\s(x)$, for
$\s=ss,uu$. The connected component containing $x$ of the intersection
of $W^{ss}(x)$ with a small $\eps$-ball centered at $x$ is the {\de
$\eps$-local stable manifold of $x$}, and is denoted by
$W^{ss}_\eps(x)$. $W^{uu}_\eps(x)$ is defined analogously.
\par
Here is another way to relax the hyperbolicity condition:
Let $f:M\to M$ be a diffeomorphism, and $K$ be an $f$-invariant set. Let $E$ and $F$ be two $Df$-invariant bundles over $K$. We say that $E\oplus F$ is a {\de $k$-dominated splitting} over $K$ if all
unit vectors $u\in E_x$ and $v\in F_x$ with $x\in K$ satisfy:
$$\|Df^k(x)u\|\leq\frac12\|Df^k(x)v\|$$
We say that $E\oplus F$ is a {\de dominated splitting} if there is $k>0$ such that $E\oplus F$ is a $k$-dominated splitting. We also say that $F$ {\de dominates} $E$.
\subsection{Blenders}\label{subsection.blenders}
%
A blender is an open set associated to a hyperbolic periodic point with some $C^1$-persistent properties, that was introduced by C. Bonatti and L. D\'{\i}az to obtain new examples of robustly transitive diffeomorphisms \cite{bodi1996}.\par
Roughly speaking,  a $cu$-blender is an open set associated to a partially hyperbolic periodic point $p$, with expanding one-dimensional center bundle, on which a convenient projection of the stable set of $p$ has topological dimension $(s+1)$, where $s$ is the dimension of $W^s(p)$, see Figure \ref{figure.blender}.\par
To be more precise, let $p$ be a partially hyperbolic periodic point such that $Df$ is expanding on $E^c$ and $\dim E^c=1$. A small open set ${\rm Bl}^{cu}(p)$ (not necessarily containing $p$) is a {\de $cu$-blender} associated to $p$
if
\begin{enumerate}
\item every {\de $(u+1)$-strip well placed in ${\rm Bl}^{cu}(p)$} transversely intersects $W^s(p)$.
\item This property is $C^1$-robust. Moreover, the open set associated to the periodic point contains a uniformly sized ball.
\end{enumerate}

 A {\de $(u+1)$-strip} is any $(u+1)$-disk containing a $u$-disk $D^{uu}$, so that $D^{uu}$ is centered at a point in ${\rm Bl}^{cu}(p)$, the radius of $D^{uu}$ is much bigger than the radius of ${\rm Bl}^{cu}(p)$, and $D^{uu}$ is almost tangent to $E^u$, i.e. the tangent vectors are $C^1$-close to $E^u$. A $(u+1)$-strip $D^{cu}$ is {\de well placed} in ${\rm Bl}^{cu}(p)$ if it is almost tangent to $E^c\oplus E^u$.\par
 Naturally, it makes sense to talk about robustness of these properties and concepts, since there is an analytic continuation of the periodic point $p$ and of the bundles $E^s$, $E^c$ and $E^u$. We can define $cs$-blenders in a similar way. For $cs$-blenders we will consider a partially hyperbolic point such that $E^c$ is one-dimensional and contracting.\par
 \begin{figure}[h]
 \begin{minipage}{6.5cm}
 \psfrag{s}{$W^s(p)$}\psfrag{p}{$p$}\psfrag{u}{$W^u(p)$}\psfrag{c}{$W^c(p)$}\psfrag{b}{${\rm Bl}^{cu}(p)$}
 \includegraphics[bb=88 308 461 592, width=6cm]{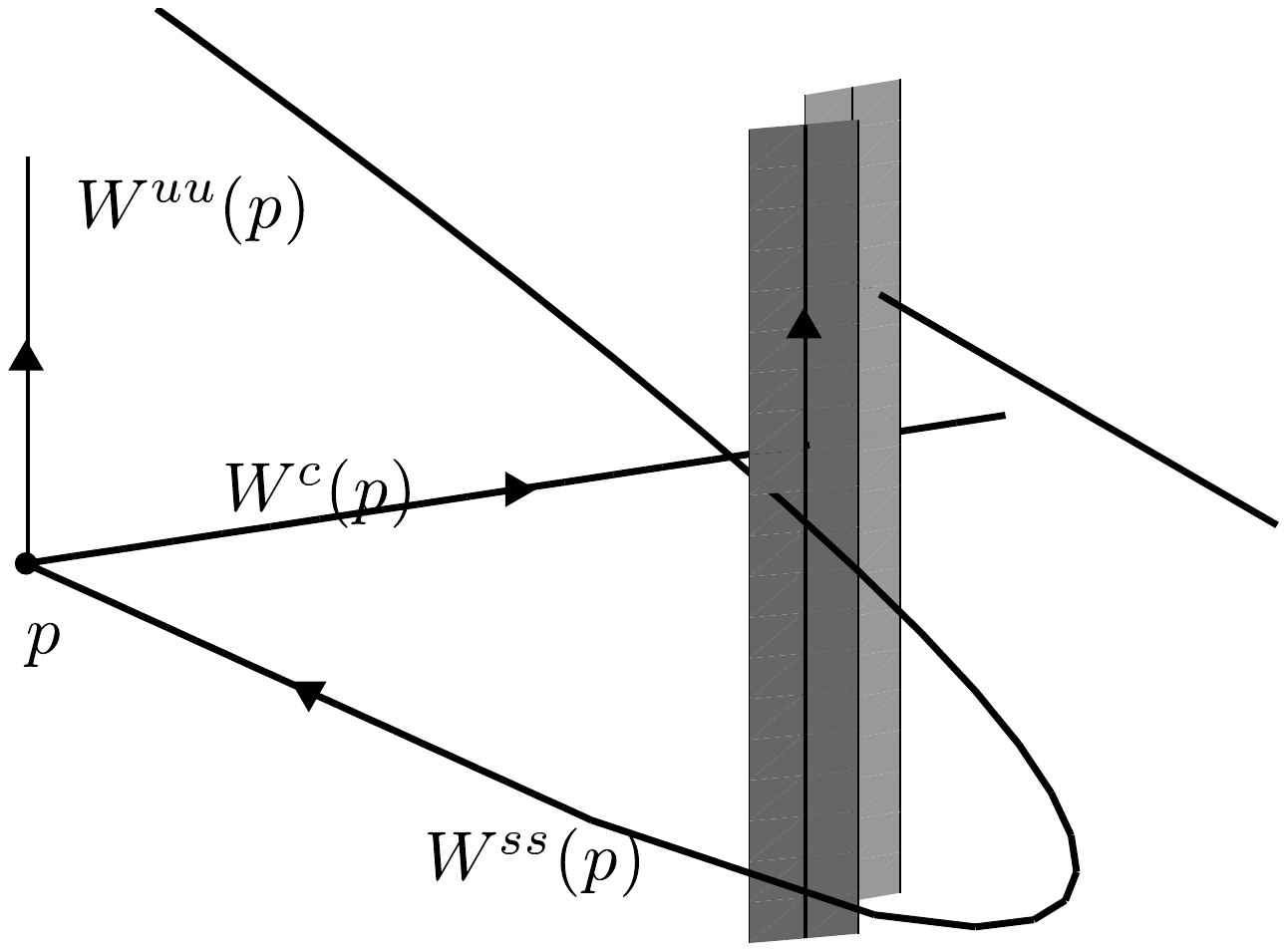}
 \end{minipage}
 \begin{minipage}{6.5cm}
 \psfrag{s}{$W^s(p)$}\psfrag{p}{$p$}\psfrag{u}{$W^u(p)$}\psfrag{c}{$W^c(p)$}\psfrag{b}{${\rm Bl}^{cu}(p)$}
 \includegraphics[bb=88 308 461 592, width=6cm]{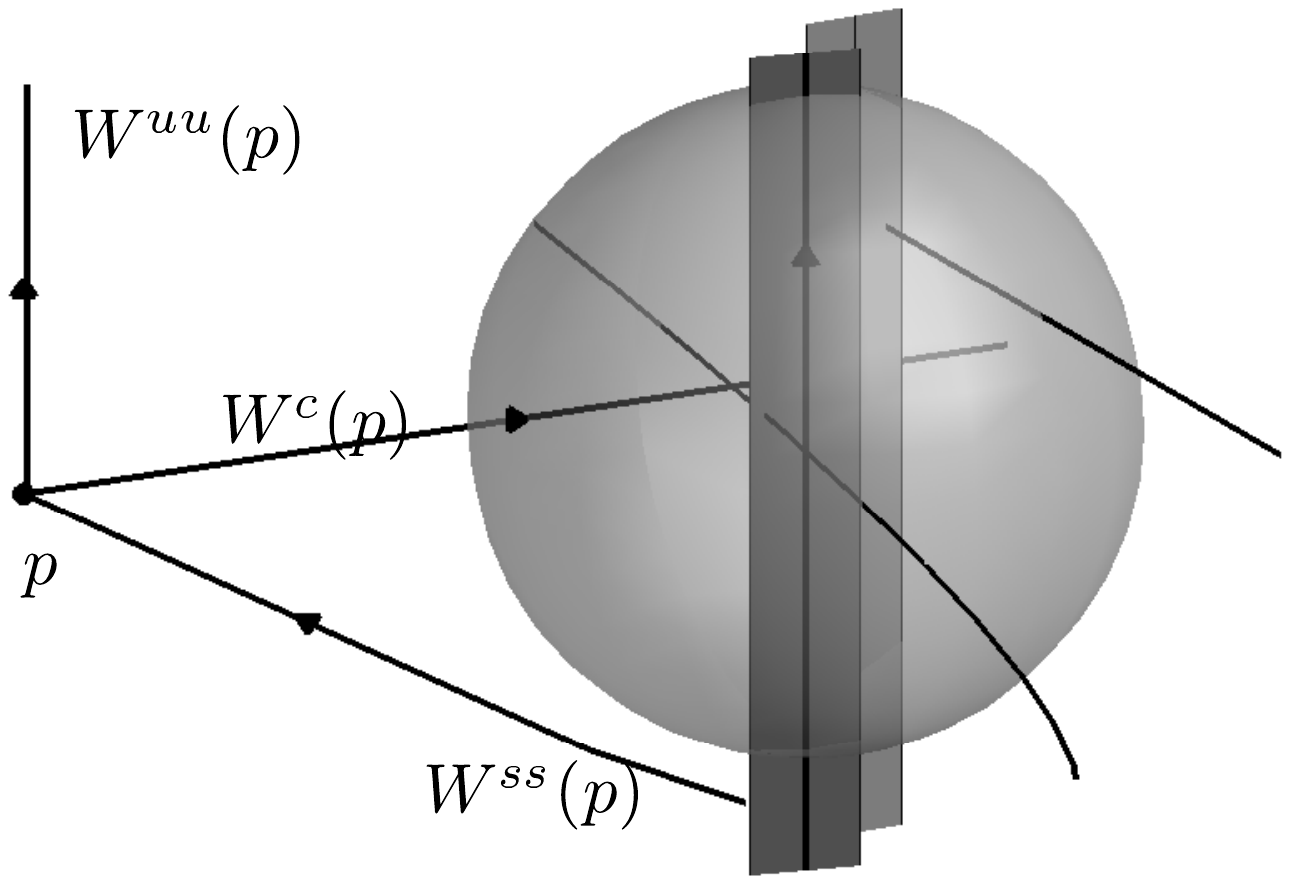}
 \end{minipage}
 \caption{\label{figure.blender} $cu$-blender associated to $p$}
 \end{figure}
 This is the definition we will be using in this work, see Definition 3.2 and Remark 3.5 of \cite{rhrhtu2009}.
 In \cite{bonattidiazvianalibro}, Chapter 6.2, there is a complete presentation on the subject and the different ways of defining blenders.\par
In Theorem 1.1. of \cite{rhrhtu2009}, we prove that $C^r$ conservative diffeomorphisms admitting blenders can be found arbitrarily near conservative diffeomorphisms with two hyperbolic periodic points whose unstable indices differ by one:
\begin{theorem}[F. Rodriguez Hertz, M. Rodriguez Hertz, A. Tahzibi, R. Ures \cite{rhrhtu2009}]  \label{proposition.blender.conservative} Let $f$ be a $C^r$ diffeomorphism preserving a smooth measure $m$ and having two hyperbolic periodic points $p$ of unstable index $(u+1)$ and $q$ of index $u$. Then $f$ is $C^1$-approximated by a $C^r$ diffeomorphism preserving $m$ which exhibits a $cu$-blender associated to the analytic continuation of $p$.
\end{theorem}

\section{A criterion for ergodicity}\label{section.criterion}
In this section, we present a general criterion to establish ergodicity and hyperbolicity of a smooth measure on a set associated to a hyperbolic periodic point. Indeed, if $p$ is a hyperbolic periodic point, let $\Lambda^s(p)$
be the set of stable Pesin manifolds which transversely intersect $W^u(p)$, and $\Lambda^u(p)$ the set of unstable Pesin manifolds which transversely intersect $W^s(p)$. We show that if these two sets have positive measure, then they coincide modulo zero, and form a hyperbolic ergodic component $\Lambda(p)$ of the measure.
No partial hyperbolicity is required along this section.\par
The proof of this criterion follows the line of the Hopf argument, and it is split up into two parts. First, it is proved that if $\Lambda(p)$ has positive measure, then it is a hyperbolic ergodic component of the measure (A.2). A more delicate proof is required to show that if $\Lambda^s(p)$ and $\Lambda^u(p)$ have positive measure then they coincide modulo a zero set (\ref{main.theorem}.1). \par
We prove (A.2) by showing that all continuous functions $\varphi:M\to\RR$ have almost constant forward Birkhoff limit $\varphi^+$ on $\Lambda(p)$. To do this, we consider two typical points $x$ and $y$ in $\Lambda(p)$ and try to see that $\varphi^+(x)=\varphi^+(y)$. Observe that $\varphi^+$ is constant on stable Pesin leaves, due to continuity of $\varphi$. Since $x,y$ are typical, $m^u_x$-a.e. point in $W^u(x)$ takes the value $\varphi^+(x)$, and $m^u_y$-a.e. point in $W^u(y)$ takes the value $\varphi^+(y)$. We may consider iterates of $x$ and $y$ so large that they are very close to $W^u(p)$. We will therefore find two disks $D_x$ and $D_y$, one contained in $W^u(f^k(x))$ and the other contained in $W^u(f^m(y))$, such that they are very close. The stable holonomy takes positive measure sets on $D_x$ into positive measure sets in $D_y$. In particular it takes the set of points in $D_x$ for which the value is $\varphi^+(x)$ into a set of positive measure in $D_y$. The fact that $\varphi^+$ is constant along stable Pesin leaves, together with the fact that $m^u_y$-a.e. point in $D_y$ has the value $\varphi^+(y)$ prove that $\varphi^+(x)=\varphi^+(y)$.\par
The proof of (A.1), requires more delicate steps. Indeed, we want to prove that a typical point $x$ in $\Lambda^u(p)$ is contained in $\Lambda^s(p)$. In order to do that, we take a typical point $y$ in $\Lambda^s(p)$. The fact that $x$ is typical implies that $x$ belongs to $\Lambda^s(p)$ if and only if $m^u_x$-a.e. point in $W^u(x)$ belongs to $\Lambda^s(p)$. Proceeding as in the previous proof, one takes suitable iterates of $x$ and $y$ so that they are very close to $W^u(p)$. We would like to follow as in the proof above, by taking holonomies between close unstable disks; however, the dimension of $W^u(x)$ might be less than $\dim (W^u(y))$. We shall therefore sub-foliate $W^u(y)$ with disks of dimension $\dim (W^u(x))$, and choose a disk $D_y\subset W^u(y)$ such that $m_D$-a.e. point in $D_y$ belongs to $\Lambda^s(p)$, where $m_D$ is the Lebesgue measure induced on $D_y$. This is possible due to a Fubini argument, since $y$ is a typical point in $\Lambda^s(p)$.
 \newline\par
Let $f$ be a $C^{1+\alpha}$ volume preserving diffeomorphism. Given a hyperbolic
periodic point $p$ let $W^u(o(p))=\bigcup_{k\in \mathbb{Z}} W^u(f^k(p))$ be the unstable manifold of the orbit of
$p$, we analogously define the stable manifold of $o(p)$.

Given a hyperbolic periodic point $p$, we define the following
sets:
\begin{eqnarray*}
\Lambda^s(p)=\{x\in\cR:\PW^s(x)\ti W^u(o(p))\neq\emptyset\}\\
\Lambda^u(p)=\{x\in \cR:\PW^u(x)\ti W^s(o(p))\neq\emptyset\}
\end{eqnarray*}
where $\ti$ means that the intersection is transversal.
$\Lambda^s(p)$, is $f$-invariant and $s$-saturated. This means that if $x\in \Lambda^s(p)$, then $W^s(x)\subset \Lambda^s(p)$. An analogous statement holds for $\Lambda^u(p)$. See Figure \ref{figure.bup}.
\begin{figure}[h]
\includegraphics[bb=91 260 543 613, width=7cm]{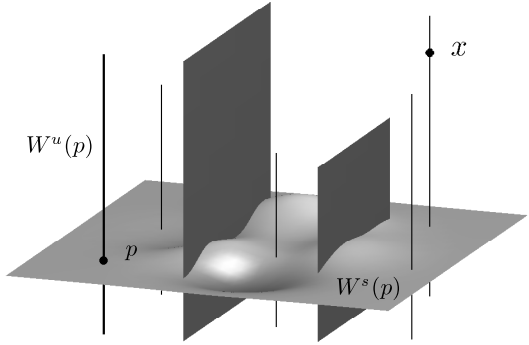}
\caption{\label{figure.bup} $\Lambda^u(p)$}
\end{figure}

We define the {\de ergodic homoclinic class} by
$$\Lambda(p):=\Lambda^u(p)\cap \Lambda^s(p)$$
The main result in this section is the following
\setcounter{maintheorema}{0}
\begin{maintheorema} Let $f:M\to M$ be a $C^{1+\alpha}$ diffeomorphism over a compact manifold $M$, and $m$ be a smooth invariant measure. If $m(\Lambda^s(p))>0$ and $m(\Lambda^u(p))>0$, then
\begin{enumerate}
\item $\Lambda(p)\circeq \Lambda^u(p)\circeq \Lambda^s(p),$
\item $\Lambda(p)$ is a hyperbolic ergodic component.
\end{enumerate}
\end{maintheorema}
Observe that, a priori, there could be points in
$\Lambda^s(p)$ or $\Lambda^u(p)$ having zero exponents, though a fortiori it will be not the case. We shall use this fact in
the proof of Theorem \ref{teorema.prueba.conjetura} (Pugh-Shub Conjecture for 2-dimensional center bundle). Observe also that we do not require that $\dim
\PW^s(x)=\dim W^s(p)$ and $\dim W^u(x)=\dim W^u(p)$, as it is seen in Figure \ref{figure.bup} although this will happen
for $m-$a.e. $x\in \Lambda(p)$. \par
Theorem \ref{main.theorem} has as a corollary:
\begin{corollary}\label{corollary}
  Let $f:M\to M$ be a $C^{1+\alpha}$ diffeomorphism and $m$ a smooth invariant measure. If $m(\Lambda(p))>0$ for a hyperbolic point $p$, then $\Lambda(p)$ is a hyperbolic ergodic component of $f$.
\end{corollary}
\subsection{Proof of (A.2)}
For any given function $\varphi\in L^1_m(M,\RR)$, let
\begin{equation}  \label{birkhoff}
\varphi^\pm(x)=\lim_{n\rightarrow\pm
\infty}\frac1n\sum^{n-1}_{i=0}\varphi(f^n(x))
\end{equation}
By Birkhoff Ergodic Theorem, the limit (\ref{birkhoff}) exists and $\varphi^+(x)=\varphi^-(x)$ a.e. $x\in M$. $\varphi^\pm(x)$ is $f$-invariant.
\begin{lemma}[Typical points for continuous functions]  \label{full.continuous} There exists an invariant set $\cT_0$ of {\de typical points} with $m(\cT_0)=1$ such that for all $\varphi\in C^0(M)$  if $x\in \cT_0$ then  $\varphi^+(w)=\varphi^+(x)$ for all $w\in W^s(x)$ and $m^u_x$-a.e. $w\in W^u(x)$.
\end{lemma}
\begin{proof}
Let us consider the full measure set:
$$\cS_0=\{x\in M: \exists \varphi^+(x)=\varphi^-(x)\}$$
For almost all $x\in \cS_0$, we have that $m^u_x$-a.e. $\xi\in W^u_{loc}(x)$, $\xi\in\cS_0$. Otherwise, there would exist a positive measure set $A\subset M$ such that for all $x\in A$ there is a subset $B_x\subset W^u_{loc}(x)\setminus\cS_0$ with $m^u_x(B_x)>0$. Considering a density point $y$ of $A$ and integrating along a transverse small disk $T$, we would obtain a set $B\subset M\setminus \cS_0$ such that
$$m(B)=\int_T m^u_x(B_x)dm_T(x)>0$$
which is an absurd. As we have seen, the following is a full measure set:
$$\cS_1=\{x\in \cS_0: m^u_x\mbox{-a.e. }\xi\in W^u_{loc}(x),\, \xi\in \cS_0\}$$
For all $x\in \cS_1$ there exists $\xi_x$ such that $m^u_x$-a.e. $\xi\in W^u_{loc}(x)$, $\varphi^+(\xi)=\varphi^-(\xi)=\varphi^-(\xi_x)=\varphi^+(\xi_x)$. But almost every $x\in \cS_1$ satisfies $\varphi^+(x)=\varphi^+(\xi_x)$. Otherwise, we would obtain a positive measure set $C\subset \cS_1$ such that $m^u_x(C\cap W^u_{loc}(x))=0$ for almost every $x$, which contradicts absolute continuity. The invariance of $\varphi^+$ yields a set $\cT_0\subset S_1$, with $m(\cT_0)=1$ and such that if $x\in\cT_0$ then $m^u_x$-a.e. $\xi\in W^u(x)$ satisfies $\varphi^+(x)=\varphi^+(\xi)$. Since $\varphi$ is continuous, we obviously have that $\varphi^+$ is constant on $W^s(x)$.
\end{proof}
Assume for simplicity that $p$ is a fixed point.
Given a continuous function $\varphi:M\to \RR$, let $\cT_0$ be the set of typical points obtained in Lemma \ref{full.continuous} and $\cR$ be the set of regular points. We will see that  $\varphi^+$ is constant on $\Lambda(p)\cap \cT_0\cap\cR$, and hence almost everywhere constant on $\Lambda(p)$. This will prove that $f$ is ergodic on $\Lambda(p)$. Hyperbolicity of the measure follows trivially.\par
For any $\eps>0$ and $l>1$ such that $m(\Lambda(p)\cap\PB)>0$, let us call $\Lambda=\Lambda(p)\cap\PB\cap \cT_0$. Without loss of generality, we may assume that all points in $\Lambda$ are Lebesgue density points of $\Lambda$, and return infinitely many times to $\Lambda$ in the future and the past. Note that there exists $\delta>0$ such that for all $x\in \Lambda$, $W^s_{loc}(x)$ contains an $s$-disc of radius $\delta$ centered at $x$.\par
Take $x,y\in \Lambda$, and consider $n>0$ such that $y_n=f^n(y)\in \Lambda$ and $d(y_n,W^u(p))<\delta/2$. We have  $W^s_{loc}(y_n)\ti W^u(p)$. \par
As a consequence of the $\lambda$-lemma, there exists $k>0$ such that $x_k=f^k(x)\in\Lambda$ and $W^u(x_k)\ti W^s_{loc}(y_n)$. See Figure \ref{figure.proofbp}. \par
\begin{figure}[h]\vspace{-.5cm}
\includegraphics[bb=91 262 543 613, width=7cm, clip]{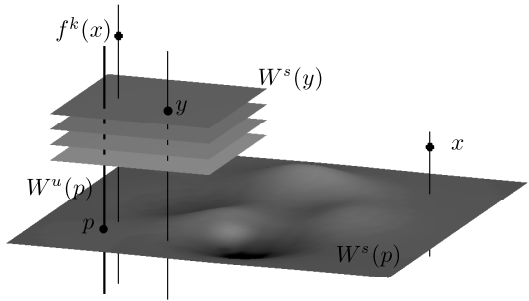}
\caption{\label{figure.proofbp} Proof of (A.1)}
\end{figure}
Since $y_n$ is a typical point for $\varphi$, for $m^u_{y_n}$-a.e. $w$ in $W^u(y_n)$ we have $\varphi^+(w)=\varphi^+(y_n)$. Also we have $\varphi^+(z)=\varphi^+(x_k)$.
Since $y_n$ is a Lebesgue density point of $\Lambda$, applying Fubini's theorem we get a point $\xi$ near $y_n$ such that the stable holonomy between $W^u_{loc}(\xi)$ and $W^u_{loc}(y_n)$ is defined for a set of $m^u_{y_n}$-positive measure in $W^u_{loc}(y_n)$. The $\lambda$-lemma implies that the stable holonomy between $W^u_{loc}(y_n)$ and $W^u_{loc}(z)$ is defined for a $m^u_{y_n}$-positive measure set, where $z$ is a point in $W^u(x_k)$, see Figure \ref{figure.proofbp}. \par
Now, $\varphi^+$ is constant along stable Pesin manifolds. And due to absolute continuity, stable holonomy takes the set of points $w$ in $W^u_{loc}(y_n)$ for which $\varphi^+(w)=\varphi^+(y_n)$ into a set of positive measure in $W^u_{loc}(z)$ for which the value of $\varphi^+$ will be $\varphi^+(y_n)$. The fact that $x_k$ is a typical point of $\varphi$ then implies that $\varphi^+(x)=\varphi^+(x_k)=\varphi^+(y_n)=\varphi^+(y)$.
\subsection{Proof of (A.1)}

In order to prove (A.1), we shall need a refinement of Lemma \ref{full.continuous}:
\begin{lemma}[Typical points for $L^1$ functions] \label{fullmeasure} Given $\varphi\in L^1$ there exists an invariant set $\cT\subset M$ of {\de typical points} of $\varphi$, with $m(\cT)=1$ such
that if $x\in\cT$ then $\varphi^+(z)=\varphi^+(x)$ for $m^s_x$-a.e. $z\in
W^s(x)$ and $m^u_x$-a.e. $z\in W^u(x)$.
\end{lemma}
\begin{proof}
Given $\varphi\in L^1$ take a sequence of continuous functions
$\varphi_n$  converging to $\varphi$ in $L^1$. Now,
$\varphi^+_{n}$ converges in $L^1$ to $\varphi^+$, so there exists
a subsequence $\varphi_{n_k}^+$ converging a.e. to
$\varphi_+$. Call $\cS$ this set of a.e. convergence. Then the set $\cT$ is the intersection of the set $\cT_0$ obtained in Lemma \ref{full.continuous} with $\cS$.
 \end{proof}
 Let us give now the proof of (A.1).
\begin{proof}[Proof of (A.1)] To simplify ideas, let us suppose that $p$ is a hyperbolic fixed point.
Let $\cT$ be the set of typical points for the characteristic function ${\mathbf 1}_{\Lambda^s(p)}$ of the set $\Lambda^s(p)$. Take $x\in \Lambda^u(p)\cap \cT$ such that all iterates of $x$ are Lebesgue density points of $\Lambda^u(p)$ and $x$ returns infinitely many times to $\Lambda^u(p)$. We shall see that $x\in \Lambda^s(p)$. This will prove $\Lambda^u(p)\overset\circ\subset\Lambda^s(p)$. The converse inclusion follows analogously.\par
Let $\eps>0$, $l>1$ be such that $m(\Lambda^s(p)\cap\PB)>0$, and let $\delta>0$ be such that for all $z\in \Lambda^s(p)\cap\PB$, the set $W^s_{loc}(z)$ contains an $s$-disc of radius $\delta>0$ centered at $z$. Consider a Lebesgue density point $y$ of $\Lambda^s=\Lambda^s(p)\cap\PB\cap\cT$ such that $d(y,W^u(p))<\delta/2$.\par
As a consequence of the $\lambda$-lemma, there exists $k>0$ such that $x_k=f^k(x)\in \Lambda^u(p)\cap \cT$ and  $W^u(x_k)\ti W^s_{loc}(y)$. Note that this intersection could a priori have positive dimension. See Figure \ref{figure.proof.criterium}.\par
\begin{figure}[h]
\includegraphics[bb= 70 261 540 613, width=7cm]{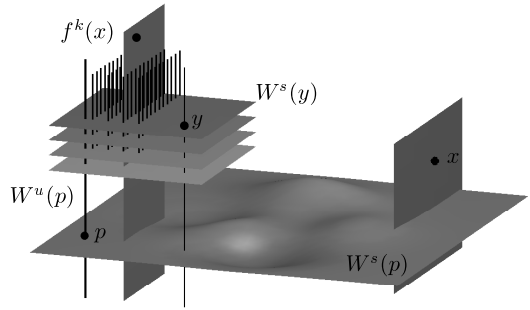}
\caption{\label{figure.proof.criterium} Proof of Theorem \ref{main.theorem}}
\end{figure}
Since $y$ is a Lebesgue density point of $\Lambda^s$, we have $m(\Lambda^s\cap B_\delta(y))>0$. Take a smooth foliation $\cL$ in $B_\delta(y)$ of dimension $u_y=n-\dim W^s_{loc}(y)$ and transverse to $W^s_{loc}(y)$. Note that $u_y\leq \dim W^u(p)$. We can also ask that the leaf $L_w$ of $\cL$ containing a point $w\in W^u(x_k)$ be contained in $W^u(x_k)$. See Figure \ref{figure.proof.criterium}.\par
By Fubini's theorem we have:
$$m(\Lambda^s\cap B_\delta(y))=\int_{W^s_{loc}(y)}m^L_\xi(L_\xi\cap\Lambda^s)dm^s_y(\xi)$$
so $m^L_\xi(L_\xi\cap \Lambda^s)>0$ for $m^s_y$-a.e. $\xi\in W^s_{loc}(y)$. Take $L\in \cL$ such that $m^L_\xi(L\cap \Lambda^s)>0$, this means that there is a $m^L_\xi$-positive measure set of points $w\in L_\xi$ such $w\in \Lambda^s(p)$. The stable holonomy takes this $m^L_\xi$-positive measure set into a $m^L_w$-positive measure set in $L_w$ for all $w\in W^u(x_k)\cap B_\delta(y)$. But $\Lambda^s(p)$ is a set saturated by stable leaves. This means that $m^L_w(L_w\cap \Lambda^s)>0$ for all $w\in W^u(f^k(x))\cap B_\delta(y)$.\par
If $W^u(x_k)\cap W^s_{loc}(y)$ is zero dimensional, this readily implies that $m^u_{x_k}(W^u(x_k)\cap \Lambda^s(p))>0$, and hence, since $x_k$ is a typical point, $x_k$ and $x$ belong to $\Lambda^s(p)$.\par
 Otherwise, take an open submanifold $T$ of $W^u(x_k)\ti W^s_{loc}(y)$. Then, by Fubini again:
$$m^u_{x_k}(\Lambda^s\cap W^u(x_k)\cap B_\delta(x))\geq \int_{T}m^L_w(L_w\cap \Lambda^s)dm_T(w)>0$$
We have that a $m^u_{x_k}$-positive measure set of $w\in W^u(x_k)$ satisfies ${\mathbf 1}_{\Lambda^s(p)}(w)=1$. Since $x_k$ is a typical point, this implies that $x_k$ and hence $x$ are in $\Lambda^s(p)$. Therefore,  $\Lambda^u(p)\overset{\circ}{\subset} \Lambda^s(p)$. The converse inclusion follows in an analogous way.
\end{proof}
\begin{remark}\label{criterio a usar}
As an immediate consequence of the $\lambda$-lemma, we have that if $W^u(p)\ti W^s(q)\neq\emptyset$ then $\Lambda^u(p)\subset \Lambda^u(q)$ and $\Lambda^s(q)\subset \Lambda^s(p)$.
\end{remark}
\subsection{An example where $m(\Lambda^u(p))=1$, $m(\Lambda^s(p))=0$ and $f$ is non-ergodic}
\label{subsection.example}
In this subsection we present an example showing the necessity of
requiring that both $m(\Lambda^u(p))>0$ and $m(\Lambda^s(p))>0$ in Theorem \ref{main.theorem}, and not just $m(\Lambda^u(p))>0$, for instance. The example closely
follows the construction by D. Dolgopyat, H. Hu and Ya. Pesin in Appendix B of \cite{barreirapesin2001}, where the authors obtain a
non-uniformly hyperbolic diffeomorphism of $\mathbb{T}^3$ with
infinitely many ergodic components. Let us show how this is obtained: \par
Consider first a diffeomorphisms $f\times id: \mathbb{T}^2\times
\mathbb{S}^1\rightarrow \mathbb{T}^2\times \mathbb{S}^1$ where
$f:\mathbb{T}^2\rightarrow \mathbb{T}^2$ is an Anosov diffeomorphism
with two fixed points $p$ and $q$, and
$id:\mathbb{S}^1\rightarrow \mathbb{S}^1$ is the identity map on the circle.\par
We will perform three small perturbations in order to obtain a diffeomorphism $g:\mathbb{T}^2\times
\mathbb{S}^1\rightarrow \mathbb{T}^2\times \mathbb{S}^1$ such that:
\begin{enumerate}
\item the tori $\mathbb{T}^2\times\{0\}$ and $\mathbb{T}^2\times\{\frac12\}$, are
$g$-invariant and the restriction of $g$ to any of them is Anosov;
\item the center Lyapunov exponent is greater than 0 almost
everywhere;
\item $g$ has two ergodic components: $\mathbb{T}^2\times
(0,\frac12)$ and $\mathbb{T}^2\times (\frac12,1)$;
\item $(p,\frac12)$ is a hyperbolic fixed point for $g$ with stable dimension
two.
\end{enumerate}
The diffeomorphism $g$ is clearly non-ergodic. Since the
center Lyapunov exponent is positive almost everywhere, we have $m(\Lambda^s(p))=0$.
We also have $m(\Lambda^u(p))=1$. Indeed, since $g$ is partially hyperbolic, all unstable manifolds grow exponentially fast. Now, almost every orbit is dense in
its ergodic component, so for almost every $x$ there is $k\in\ZZ$ such that
$W^u(g^k(x))\pitchfork W^s(p)\neq \emptyset$. Since $W^s(p)$ is invariant this implies that
$W^u(x)\pitchfork W^s(p)\neq \emptyset$.\par
Let us describe the perturbations we need to do in order to
obtain $g$.
We begin by taking a perturbation $f_1$ of $f$ as in the work of M. Shub and A. Wilkinson
(see \cite{shubwilkinson2000} and Appendix B of \cite{barreirapesin2001}):\par
Choose $f_1=f\circ j$, where the support of $j$ is contained in two small open sets
$U_1\subset \mathbb{T}^2\times (0,\frac12)$ and $U_2\subset
\mathbb{T}^2\times (\frac12,1)$. Take $U_1$ and $U_2$ so that they do
not intersect a small neighborhood of the circles $\{p\}\times \mathbb{S}^1$ and $\{q\}\times \mathbb{S}^1$.
The effect of this perturbation is that
$$
\int_{\mathbb{T}^2\times
(0,\frac12)}\log ||Df_1|_{E^c}|| dm>0\quad \mbox{and}\quad\int_{\mathbb{T}^2\times
(\frac12,1)}\log ||Df_1|_{E^c}|| dm>0$$
We will now perturb $f_1$ in a small neighborhood of
$\{p\}\times \mathbb{S}^1$ to obtain a diffeomorphism $f_2$. We want the dynamics of $f_2$
in $\{p\}\times \mathbb{S}^1$ to be
Morse-Smale with only two fixed points $(p,0)$ and $(p,\frac12)$. We want the resulting $f_2$ to be volume preserving.
In order to obtain this we will define, as in Appendix B of \cite{barreirapesin2001}, a divergence free vector
field $X$ supported in a small neighborhood of $\{p\}\times
\mathbb{S}^1$. Call $x,y$ the coordinates of $\mathbb{T}^2$ in a neighborhood of $p$
and $z$ the coordinate in $\mathbb{S}^1$.\par
The divergence free vector field we take is:
$$X(x,y,z)=\left( -\pi\cos(2\pi z)\psi'(y)\psi(x),
-\pi\cos(2\pi z)\psi'(x)\psi(y),\psi'(x)\psi'(y)\sin(2\pi z)\right)$$

where $\psi':\RR\rightarrow \mathbb{R}$ is an even bump function
such that $\psi'\equiv 1$ on a small interval $(-\eps,\eps)$, $\psi'(t)\equiv 0$ for $|t|>\eps_0$ with $\eps<\eps_0$, and $\psi(0)=\psi(\eps_0)=0$.\par
We obtain $f_2$ by composing $f_1$ with the time $t$ map of the flow generated
by $X$ for very small positive $t$. Observe that, since the
vector field is tangent to $\mathbb{T}^2\times \{0\}$ and
$\mathbb{T}^2\times \{\frac12\}$, these tori remain invariant by
$f_2$.\par
Finally, call $\Gamma(f_2)$ the union
of all non-open accessibility
classes contained in $\mathbb{T}^2\times [0,\frac12]$. By \cite{rhrhu2006} this set is a codimension-one
lamination tangent to $E^s\oplus E^u$, such that each lamina is an
accessibility class. There is a natural projection from $\Gamma(f_2)$ to
$\mathbb{T}^2\times \{0\}$ along the center leaves. This projection
is a covering projection when restricted to a single lamina. Then the
laminae are planes, cylinders or tori. In \cite{rhrhu2006} it is
proved that the laminae for which there is an open arc in the complement of
$\Gamma(f_2)$ with one endpoint in $\Gamma(f_2)$ are periodic, and the set of periodic points is dense in each lamina (even with the
intrinsic topology induced by the metric of the ambient manifold).
Since the dynamics of $f_2$ restricted to each lamina is Anosov, and planes and cylinders do not support such dynamics, then these leaves are tori. Now, we cannot have invariant tori
different from ${\mathbb T}^2\times\{0\}$ or ${\mathbb T}^2\times\{\frac12\}$, because they would intersect
$\{p\}\times \mathbb{S}^1$. So, we have a dichotomy: either
$\Gamma(f_2)=\mathbb{T}^2\times \{0\}\cup \mathbb{T}^2\times
\{\frac12\}$ or $\Gamma(f_2)=\mathbb{T}^2\times [0,\frac12]$. In the
first case we are done because this would imply that
$\mathbb{T}^2\times (0,\frac12)$ is an accessibility class and,
moreover, an ergodic component. In the second case,
we can use the Unweaving Lemma (Lemma A.4.3 of \cite{rhrhu2006}) to
obtain an open accessibility class perturbing $f_2$ in an arbitrary small
neighborhood of any fixed point of $\{q\}\times (0,\frac12)$. In this way we obtain a new diffeomorphism $f_3$. Since the above
mentioned dichotomy remains valid and there is an open accessibility
class we obtain that $\Gamma(f_3)=\mathbb{T}^2\times \{0\}\cup
\mathbb{T}^2\times \{\frac12\}$. The desired diffeomorphism $g$ is
$f_2$ in case $\Gamma(f_2)$ verifies the first equality of the
dichotomy, or $f_3$ in the other case. Finally, observe that since
$\mathbb{T}^2\times (0,\frac12)$ is an ergodic component and our
perturbations have been small enough to have that
$\int_{\mathbb{T}^2\times (0,\frac12)}\log ||Dg|_{E^c}|| dm >0,$
the center Lyapunov exponent is positive almost
everywhere.

\section{Proof of theorem \ref{teorema.prueba.conjetura}}\label{section.prueba.conjetura}
Let $f$ be a $C^{1+\alpha}$ partially hyperbolic diffeomorphism with two-dimensional center bundle, preserving a smooth measure $m$. We want to show that $f$ can be $C^1$-approximated by stably ergodic diffeomorphisms. As we stated in Section \ref{section.sketch.pugh.shub}, we shall prove the following:
\begin{proposition}\label{theorem.approximation}
Let $f$ be a $C^{1+\alpha}$ partially hyperbolic diffeomorphisms with two-dimensional center bundle, preserving a smooth measure $m$. Then either:
\begin{enumerate}
\item $f$ can be $C^1$-approximated by stably ergodic diffeomorphisms, or
\item $f$ can be $C^1$-approximated by diffeomorphisms satisfying the hypotheses of Theorem \ref{teorema.establemente.ergodico}.
\end{enumerate}
\end{proposition}
The Conjecture of C. Pugh and M. Shub will follow by proving:
\setcounter{maintheorema}{2}
\begin{maintheorema}
Let $g$ be a $C^{1+\alpha}$ partially hyperbolic diffeomorphism preserving a smooth measure, and such that its center bundle is two-dimensional. Assume $f$ satisfies the following properties:
 \begin{enumerate}
 \item $g$ is {\em accessible},
 \item  $E^c = E^- \oplus E^+$ admits a nontrivial dominated splitting
  \item $\int_M \lambda^- d m < 0 $ and $\int_M \lambda^+ d m > 0,$
 \item $g$ admits a $cs$-blender ${\rm Bl}^{cs}(p)$ and a $cu$-blender ${\rm Bl}^{cu}(p)$ associated to a hyperbolic periodic point of stable index $(s+1)$.
  \end{enumerate}
  then $g$ is stably ergodic.
\end{maintheorema}
We devote Section \S\ref{subsection.approximation} to proving Proposition \ref{theorem.approximation} and Section \S\ref{subsection.establemente.ergodico} to proving Theorem \ref{teorema.establemente.ergodico}.
\subsection{Proof of Proposition \ref{theorem.approximation}} \label{subsection.approximation}
Let us show that either $f$ is already known to be $C^1$-approximated by stably ergodic diffeomorphisms or else $f$ is $C^1$-approximated by diffeomorphisms in the hypotheses of Theorem \ref{teorema.establemente.ergodico}.\par
The following theorem implies that $f$ is $C^1$-approximated by open sets of diffeomorphisms satisfying hypothesis (C.1):
\begin{theorem}[D. Dolgopyat, A. Wilkinson \cite{dolgopyat-wilkinson2003}] Stable accessibility is $C^1$-dense among partially hyperbolic diffeomorphisms, volume preserving or not.
\end{theorem}
So, we may consider from now on a volume preserving $C^{1+\alpha}$ partially hyperbolic diffeomorphism $f_1$ that is stably accessible and is $C^1$ close to $f$.\newline\par
Let us assume that $f_1$ does not admit a non-trivial dominated splitting of its center bundle.
Call, for all $g$ $C^1$-near $f_1$,
$\lambda_g^-(x)\leq\lambda_g^+(x)$ the central Lyapunov exponents of
$x$ with respect to $g$, and recall that
$$\int_M (\lambda_g^-(x)+\lambda^+_g(x))\,dm (x)=\int_M \log Jac (Dg(x)|E^c_x)\,dm(x)$$
Observe that this amount depends continuously on $g$, due to
continuity of $E^c_x$ with respect to $g$. We lose no generality in assuming that:
$$\int_M(\lambda^+_{f_1}(x)+\lambda^-_{f_1}(x))\,dm(x)\geq 0$$
We use the following
\begin{theorem}[A. Baraviera, C. Bonatti
\cite{baraviera-bonatti2003}]\label{teo.baraviera.bonatti}
 Let
$f$ be a $C^1$ partially hyperbolic diffeomorphism, then there are
arbitrarily small $C^1$-perturbations $g$ of $f$ such that
$$\int_M\log Jac(Dg(x)|E^c_x)dm(x) >\int_M\log Jac(Df(x)|E^c_x)dm(x). $$
\end{theorem}
\noindent to obtain a $C^{1+\alpha}$ partially hyperbolic diffeomorphism $g$ preserving $m$, $C^1$ close to $f_1$ such that
\begin{equation}\label{suma.expoentes.positiva}
 \int_M ( \lambda^{-}_{g} (x) +  \lambda^{+}_{g} (x))\, dm (x) > 0.
\end{equation}
Assume $g$ does not admit a non-trivial dominated splitting $E^c=E^-\oplus E^+$, otherwise we would have obtained a diffeomorphism satisfying (C.1) and (C.2).
Let $D(g,k)$ be the set of points $x$ such that there is a
nontrivial $k$-dominated splitting of $E^c$ along the orbit of $x$. And denote
$D(g)=\bigcup_{k=1}^\infty D(g,k)$, the set of points such that there is a non-trivial dominated splitting along their orbits. \par
Recall that the hypothesis of accessibility implies that the orbit
of almost every point is dense \cite{burns-dolgopyat-pesin2002}.
Hence, $m(D(g,k))=0$ for every $k>0$, otherwise $D(g,k)$ would contain a dense orbit and, since it is a closed invariant set it would be $M$, $g$ would then admit a dominated splitting for $E^c$.
We obviously have $m(D(g))=0$, so by
Proposition 4.17 of \cite{bochi-viana2005}, we can
conclude that for each $\delta>0$ there is a $C^{1+\alpha}$ conservative
 diffeomorphism $h$ arbitrarily $C^{1}$-close to $g$ such that
\begin{align*}
\int_{M} \lambda^-_h (x) dm(x)
 &\geq  \int_M \lambda^-_g(x) dm(x) + \frac 12\int_{M\setminus D(g)}( \lambda^+_g( x) -
 \lambda^-_g( x))dm(x)  - \delta  \\
&=
 \frac12\int_{M} ( \lambda^+_g( x) +  \lambda^-_g(x)) dm(x) - \delta
 \end{align*}
Choosing a suitable $\delta>0$, we obtain $h$ such that $ \int_M  \lambda^-
_h(x)dm(x)>0 $. Hence, there exists a subset $A$ of $M$
with positive Lebesgue measure such that $\lambda^-_h (x) > 0$ for
all $x \in A$. And $\lambda^+ _h(x)\geq \lambda^-(x) > 0$ for $ x \in A$. Recall that $h$ has the accessibility property. Theorem \ref{BDP} then implies that $h$ is stably ergodic. \newline\par
We have found a $C^{1+\alpha}$ conservative partially diffeomorphism $f_2$, $C^1$ close to $f$, satisfying (C.1) and (C.2). Moreover, $f_2$ is stably accessible.
Note that (\ref{suma.expoentes.positiva}) implies that $\int_M\lambda^+_{f_2}(x)\,dm(x)>0$.
Now, either $f_2$ satisfies (C.3), or else $\int_M\lambda^-_{f_2}(x)\geq 0$.
Theorem \ref{teo.baraviera.bonatti} above applied to the bundle $E^-$ implies that there is a
$C^1$ perturbation $g$ of $f_2$ so that $\int_M\lambda^-_g(x)\,dm(x)>0$. Theorem \ref{BDP} applies, and we obtain that $g$ is stably ergodic. \newline\par
So far, we have found a $C^{1+\alpha}$ diffeomorphism $f_3$ that is $C^1$ close to $f$ such that either $f_3$ is $C^1$-approximated by stably ergodic diffeomorphisms or $f_3$ is in a neighborhood of diffeomorphisms satisfying (C.1), (C.2) and (C.3).
We want to find a $C^1$-perturbation $h\in C^{1+\alpha}$ of $f_3$, preserving $m$ and admitting a $cs$-blender ${\rm Bl}^{cs}(p)$ and a $cu$-blender ${\rm Bl}^{cu}(p)$ associated to a hyperbolic periodic point $p$ of stable index
$(s+1)$. Let us begin by proving:
\begin{lemma} \label{ergodicclosing}
Let $f_3$ be a $C^{1+\alpha}$ partially hyperbolic diffeomorphisms satisfying hypotheses (C.1), (C.2) and (C.3). Let us further assume that $f_3$ is stably accessible. Then,
either
\begin{enumerate}
\item $f_3$ is $C^1$-approximated by a stably ergodic diffeomorphism,
or \item $f_3$ is $C^1$-approximated by $g\in C^{1+\alpha}$ preserving $m$ and having
three hyperbolic periodic points with unstable indices $u$, $(u+1)$ and $(u+2)$, respectively,
where $u=\dim E^u$.
\end{enumerate}
\end{lemma}
\begin{proof}
Let $E^c=E^-\oplus E^+$ be the dominated splitting for the center bundle of $f_3$.
If $E^-$ or $E^+$ were hyperbolic, then $f_3$ could be seen as a
partially hyperbolic diffeomorphism with one-dimensional center bundle, and hence $f_3$ would be
known to be approximated by stably ergodic diffeomorphisms, see \cite{burns-wilkinson2006}, \cite{rhrhu2006}.\par
Assuming that $E^-$ nor $E^+$ are hyperbolic, we can find points $x^+_n,x^-_n\in M$ and sequences of integers $k_n,l_n\uparrow \infty$ such that
$$\frac1{k_n}\log\|Df^{k_n}(x^+_n)|E^+_{x_n}\|\geq-\frac{1}{n}\quad\mbox{and}\quad \frac1{l_n}\log\|Df^{l_n}(x^-_n)|E^-_{x_n}\|\leq\frac{1}{n}$$
Let us call
$\mu^+_n=\frac{1}{k_n}\sum_{j=0}^{k_n-1}\delta_{f^j(x^+_n)}$ and $\mu^-_n=\frac{1}{l_n}\sum_{j=0}^{l_n-1}\delta_{f^j(x^-_n)}$ the
probability measures supported, respectively, on the pieces of orbit
$\{f^j(x^+_n)\}_{j=0}^{k_n-1}$ and $\{f^j(x^-_n)\}_{j=0}^{l_n-1}$. Then there are subsequences, which
we continue to call $\mu^+_n, \mu^-_n$, such that $\mu^+_n\to\mu^+$ and $\mu^-_n\to\mu^-$ in
the weak-topology. We have
$$\int_M\lambda^+(x)d\mu^+(x)=\int_M\log\|Df_x|E^+_x\|d\mu^+(x)\leq 0
$$ Let $\mu^+_*$ be an ergodic component of $\mu^+$ such that
$\int_M\lambda^+(x)d\mu^+_*(x)\leq 0$. And let us recall the following:
\begin{theorem}[Ergodic Closing Lemma \cite{arnaud1998}]\label{ergodiclosing}
Consider a diffeomorphism $f$ preserving a smooth volume $m$. Then there is an
$f$-invariant set $\Sigma(f)$, the set of {\de well closable points}, such that:
\begin{enumerate}
\item $\mu(\Sigma(f) ) = 1$ for any invariant probability measure
$\mu$.
\item For every $x \in \Sigma (f)$ and $ \eps > 0$ there is a $C^1$-perturbation $g\in C^{1+\alpha}$ preserving $m$  such that
$x $ is a periodic point of $g$ and
 $d (f^i(x), g^i(x)) < \eps$ for all $i \in [0, \pi_g(x)],$ where $\pi_g(x)$ is the period of $x$ with respect to
 $g$.
\end{enumerate}
\end{theorem}

By Theorem \ref{ergodiclosing}, we have that $\mu^+_*$-a.e. $x$ is well
closable, and satisfies
$\lambda^+(x)=\int_M\lambda^+(x)d\mu^+_*(x)\leq0$, so there are $C^{1+\alpha}$
diffeomorphisms $g$ $C^1$-close to $f_3$ preserving $m$, and $g$-periodic points $p$ for which $\lambda^+_g(p)$ is close to be non-positive. If $\lambda^+_g(p)<0$, then $p$ is a hyperbolic periodic point of unstable index $u$. Otherwise,
$\lambda^+_g(p)$ is close to $0$. Let us recall the following:
\begin{proposition}[Conservative version of Franks
Lemma \cite{bdp2003}]\label{franks} Let $f$ be a diffeomorphism preserving a smooth measure $m$, $p$ be a
periodic point. Assume that $B$ is a conservative
$\eps$-perturbation of $Df$ along the orbit of $p$. Then for every neighborhood $V$
of the orbit of $p$ there is a $C^1$-perturbation $h\in C^{1+\alpha}$ preserving $m$ and coinciding with $f$ on the orbit of $p$ and out of $V$,
such that $Dh$ is equal to $B$ on the orbit of $p$.
\end{proposition}
If $\lambda^+_g(p)$ is close to $0$, then Proposition \ref{franks} allows us to find some
 $C^1$-perturbation which is a $C^{1+\alpha}$ diffeomorphism preserving $m$, for which the unstable index of $p$ is $u$. In
an analogous way, we find $h$ and $q$ such that $p$ and
$q$ are hyperbolic $h$-periodic points of unstable indices $u$ and $(u+2)$ respectively.\par
We are left to find a periodic point of index $(u+1)$.
Let $A=\{x:\lambda^+_h(x)>0\}$, then by assumption $m(A)>0$.
If there is $B\subset A$ with $m(B)>0$ such that $\lambda^-_h(x)>0$ on $B$, then by
Theorem \ref{BDP} we have that $h$ is stably ergodic. Otherwise, we have that $m$-a.e. $x$ in
$A$ satisfies $\lambda^-(x)\leq 0$. We can therefore choose an
ergodic component $\mu$ of $m_A$ such that
$\int_M\lambda^+(x)d\mu(x)>0$ and $\int_M\lambda^-(x)d\mu(x)\leq 0$. The
technique described above allows us to find a $C^{1+\alpha}$ volume preserving diffeomorphism, $C^1$-close to $f_3$ and $p_0,p_1,p_2$ such that $p_i$ are periodic points with unstable
indices $(u+i)$, for $i=0,1,2$. This finishes the proof of Lemma \ref{ergodicclosing}.\end{proof}
So far, we may assume that we have a $C^{1+\alpha}$ stably accessible partially hyperbolic diffeomorphism $f_4$, preserving $m$, $C^1$-close to $f$, such that $f_4$ satisfies (C.1), (C.2) and (C.3), and has three hyperbolic periodic points $p_0$, $p_1$ and $p_2$ of unstable indices $u$, $(u+1)$ and $(u+2)$ respectively.
Applying Theorem \ref{proposition.blender.conservative} we obtain a $C^1$-perturbation $f_5\in C^{1+\alpha}$ preserving $m$ and having a $cs$-blender ${\rm Bl}^{cs}(p_1)$ and a $cu$-blender ${\rm Bl}^{cu}(p_1)$ associated to the analytic continuation of $p_1$. This ends the proof of Proposition \ref{theorem.approximation}.
\subsection{Proof of Theorem \ref{teorema.establemente.ergodico}}\label{subsection.establemente.ergodico}
Let us define
$$ C^+ = \{ x \in M : \lambda^{+}(x) > 0\}, $$
$$ C^- = \{ x \in M : \lambda^{-}(x) < 0\}. $$
By hypothesis (C.3),  both $C^{+}$ and $C^{-}$ have positive Lebesgue measure.
Now, if $x$ is a regular point such that $x\notin C^+$, then $\lambda^+(x)\leq 0$. Since $\lambda^\pm(x)$ are the Lyapunov exponents associated to $E^\pm(x)$, and $E^c=E^-\oplus E^+$ is a dominated splitting, we have $\lambda^-(x)<0$. So, $x\in C^-$. Therefore $ C^+ \cup C^- \circeq M$.\par
Recall that, as we mentioned in Remark \ref{remark.accessibility.dense}, almost every orbit is dense, since we have (C.1). Therefore, we shall assume without loss of generality that all points in $C^+$ and $C^-$ are regular, and their orbits are dense.\par
Let us show that $C^+\subset\Lambda^u(p)$ and $C^-\subset \Lambda^s(p)$. Then, by hypothesis (C.3), we shall have that $m(\Lambda^u(p))>0$ and $m(\Lambda^s(p))>0$. The criterion established in Theorem \ref{main.theorem} then shows that $\Lambda^s(p)\circeq\Lambda^u(p)\circeq \Lambda$, and the previous comment shows that $f$ is ergodic. We shall afterwards explain how this argument is adapted to show that, in fact, $f$ is stably ergodic.\par
Let $x\in C^+$. Then the unstable Pesin manifold contains a $(u+1)$-disk $D$ tangent to $E^+\oplus E^u$. Since each orbit of $C^+$ is dense, there is an iterate $n>0$ such that $x_n=f^n(x)$ belongs to the $cu$-blender ${\rm Bl}^{cu}(p)$. The disk $D_n=f^n(D)$ is well placed, but it is not necessarily a $(u+1)$-strip, since its size could be not adequate yet. However, since there is uniform expansion along the strong unstable leaves $W^{uu}(x)$, there is an eventual iterate $k>0$ such that $x_k$ belongs to ${\rm Bl}^{cu}(p)$ {\em and} $D_k$ is a $(u+1)$-strip well placed in ${\rm Bl}^{cu}(p)$. Therefore, $D_k\pitchfork W^s(p)$, and hence $x_k\in\Lambda^u(p)$. Invariance implies $x\in\Lambda^u(p)$. In an analogous way, we show that $C^-\subset \Lambda^s(p)$.\par
Now, let $\eps>0$ be such that for all $g$ in a $C^1$-neighborhood of $f$, the $cu$-blender associated with the analytic continuation of $p$, ${\rm Bl}^{cu}(p_g)$, contains a ball of radius $\eps$. And consider any $g\in C^{1+\alpha}$ preserving $m$, so $C^1$-close to $f$ that $m$-almost every orbit of $g$ is $\eps/2$-dense (see Remarks \ref{remark.accessibility.dense} and \ref{remark.eps.accessibility}). We shall see that all such $g$ are ergodic. This proves that $f$ is stably ergodic.\par
Let us note that $g$ satisfies hypotheses (C.2) and (C.3). Indeed, (C.2) is an open property and, since $\int_M \lambda^+(x)dm(x)=\int_M\log Jac(Dg(x)|E^+_x)dm(x)$, the fact that (C.2) holds, together with the fact that $E^+_x$ varies continuously with respect to $g$, imply that (C.3) is open. Therefore, we have that $m(C^+_g)>0$ and $m(C^-_g)>0$ and by the same argument as above, $C^+_g\cup C^-_g\circeq M$.\par
To finish, observe that if $x\in C^+_g$, then the fact that the orbit of $x$ is $\eps/2$-dense implies that there are arbitrarily large iterates of $x$ belonging to the $cu$-blender ${\rm Bl}^{cu}(p_g)$. Proceeding as above, it is now easy to show that $x\in \Lambda^u(p)$. In analogous way it follows that $C^-\subset\Lambda^s(p)$ and hence, $g$ is ergodic.

\end{document}